\documentclass[11pt, twoside]{article}
\usepackage{amsfonts,amssymb,amsmath,amsthm}
\usepackage{graphicx}
\usepackage[all]{xy}
\usepackage{multirow} 
\usepackage{psfrag,xmpmulti,amscd,color,pstricks, import}
 
 \divide\oddsidemargin by 2
\newtheorem{thm}{Theorem}[section]
\newtheorem{cor}[thm]{Corollary}
\newtheorem{lem}[thm]{Lemma}
\newtheorem{prop}[thm]{Proposition}

\newtheorem{lemma}[thm]{Lemma}

\theoremstyle{definition}
\newtheorem{defn}[thm]{Definition}
\newtheorem{rem}[thm]{Remark}
\newtheorem*{rem*}{Remark}

\numberwithin{equation}{section}

\definecolor{OrangeRed}{cmyk}{0,0.6,1,0}            
\definecolor{DarkBlue}{cmyk}{1,1,0,0.20}
\definecolor{DarkGreen}{cmyk}{1,0,0.6,0.2}
\definecolor{myblue}{rgb}{0.66,0.78,1.00}
\definecolor{Violet}{cmyk}{0.79,0.88,0,0}
\definecolor{Lavender}{cmyk}{0,0.48,0,0}

\renewcommand{\Im}{\operatorname{Im}}
\renewcommand{\Re}{\operatorname{Re}}
\newcommand{\im}{\operatorname{Im}}
\newcommand{\re}{\operatorname{Re}}
\renewcommand{\arg}{\operatorname{arg}}

\newcommand{\dist}{\operatorname{dist}}

\newcommand{\Id}{\operatorname{Id}}

\renewcommand{\AA}{{\mathcal A}}

\newcommand{\MM}{{\mathcal M}}

\renewcommand{\SS}{{\mathcal S}}

\newcommand{\XX}{\mathcal X}

\newcommand{\C}{{\mathbb C}}

\newcommand{\D}{{\mathbb D}}
\renewcommand{\H}{\mathbb{H}}

\newcommand{\N}{{\mathbb N}}

\renewcommand{\P}{{\mathbb P}}
\newcommand{\R}{{\mathbb R}}

\newcommand{\T}{{\mathbb T}}

\newcommand{\ra}{\rightarrow}
\newcommand{\ov}{\overline}

\renewcommand{\emptyset}{\varnothing}
\renewcommand{\epsilon}{\varepsilon}
\renewcommand{\phi}{\varphi}

\title{Escaping Fatou components with disjoint hyperbolic limit sets}
 
\vspace{5cm}

\author{
Veronica Beltrami, Anna Miriam Benini, Alberto Saracco \thanks{This work was partially supported by   the Indam Groups Gnampa and GNSAGA and by PRIN 2022 Real and Complex Manifolds: Geometry and Holomorphic Dynamics. A.S. thanks Franc Forstneri\v c for a useful discussion.
}
}

\begin{document}

\maketitle  
\begin{abstract}
We construct  automorphisms of $\C^2$ with  a cycle of escaping Fatou components, on which there are exactly two limit functions, both of rank 1. On each such  Fatou component, the limit sets for these limit  functions are two disjoint hyperbolic subsets of the line at infinity. 
\end{abstract}

\section{Introduction}
Transcendental H\'enon maps are  automorphisms of $\C^2$  with  constant Jacobian of the form 
$$
F(z,w):=(f(z)-\delta w,z) \text{ with $f:\C\ra\C$ entire transcendental}.
$$ 

In analogy with classical complex H\'enon maps, for which $f$ is assumed to be a polynomial, the dynamical investigation of transcendental H\'enon maps can rely on tools and knowledge from one dimensional complex dynamics, which is better understood  than its higher dimensional counterpart. They have been introduced in \cite{Dujardin04}. 
General properties of transcendental H\'enon maps were established  in \cite{henon1},\cite{henon2},\cite{henon3} and examples with interesting dynamical features were presented. 

Let $\P^2$ be the complex projective space obtained by compactifying $
\C^2$ by adding the line at infinity $\ell_\infty$. We define the Fatou set of $F$ as the set of points in $\C^2$ near which the iterates form a normal family with respect to the complex structure induced by   $\P^2$ (compare with \cite{henon1}, section 1). A Fatou component is a connected component of the Fatou set. Given a Fatou component $\Omega$ we call a function $h:\Omega\ra\P^2$ a  \emph{limit function} for $\Omega $  if there exists a subsequence $n_k$  such that   $F^{n_k}\ra h$ uniformly on compact subsets of $\Omega$. The image $h(\Omega) $ of a limit function $h$ is called a \emph{limit set} (for $\Omega$). By Lemma 4.3 and 2.4 in \cite{henon1}, each limit sets is either contained in $\C^2$ or contained in $\ell_\infty$.

In this paper we investigate \emph{escaping Fatou components}, that is Fatou components for  which all  limit sets lie in the line at infinity.

More precisely we construct a transcendental H\'enon map with a cycle of escaping Fatou components satisfying the following properties. Let $\H$ denote the right half plane, $-\H$ denote the left half plane. 

\begin{thm}\label{Main thm}
Let 
$$F(z,w):=(e^{-z^2}+e^{\pi i }\delta w,z),\ \delta>2.$$ 
Then $F$ has a cycle of four Fatou components $\Omega^{ab}$ with $a,b\in\{+,-\}$, each of which is biholomorphic to $\H\times\H$. There are exactly two limit functions $h_1,h_2$, both of rank 1,  such that
 $$
h_1(\Omega^{aa})=h_2(\Omega^{a(-a)})=\H  \text{ and } h_1(\Omega^{a(-a)})=h_2(\Omega^{aa})=-\H \text{ for all $a$.}
 $$
 Moreover, $F$ is conjugate to its linear part on every $\Omega^{ab}$.
\end{thm}

The main points of interest of this result are that the limit functions have rank one, that each Fatou component has two disjoint limit sets (compare \cite{JL04} for restrictions on the presence of several limit sets), and that the limit sets $\H, -\H$ are hyperbolic. 

For general automorphisms of $\C^2$ there are very few examples of limit functions of rank 1 (\cite{JL04}, \cite{BTBP}), and for polynomial H\'enon maps, it is not even known whether rank 1 limit functions  can exist;  in fact, their existence has  been excluded provided the Jacobian is small enough (\cite{LP14}). On the other hand, they are abundant for holomorphic \emph{endomorphisms} of $\C^2$ (\cite{BTFP15}, Theorem 4). 
For transcendental H\'enon maps, rank 1 limit functions  seem to appear  naturally for escaping Fatou components (\cite{BSZ}).  To our knowledge there were no previous examples of hyperbolic limit sets for automorphisms of $\C^2$. One possible reason for the natural appearance of these phenomena might be that $F$ is not defined on $\ell_\infty$, hence there is no natural dynamics on limit sets contained there. 

One can see F as a special case of maps of the form 
$$
F(z,w):=(e^{-z^k}+e^{\frac{2\pi i }{k}}\delta w,z), \text{ $\delta>2,\ k\in \N$}
$$
 Analogous results hold for such maps, and are proven   in \cite{Beltrami} with similar techniques.

\section{Proof of Theorem~\ref{Main thm} }
From now on let $F$ be as in Theorem~\ref{Main thm}, 
\begin{equation}\label{eq:F}
F(z,w)=(e^{-z^2}-\delta w,z) \text{ with $\delta>2$.}
\end{equation}
Throughout the paper, given a point   $P=(z_0,w_0)\in \C^2$ and $n\in\N$  we denote its iterates by  $F^n(P)=:  (z_n,w_n)$. 

\subsection{Computing limit functions}\label{sec:computing limit functions}
In this section we give an explicit expression for the iterates of $F$ and their formal limit.  A direct computation (compare \cite{BSZ}) shows that 
\begin{align*}
F^{2n}(z_0,w_0)&= (-\delta)^n\Big(z_0+\sum_{j=1}^n(-\delta)^{-j}f(z_{2j-1}),w_0+\sum_{j=1}^n(-\delta)^{-j}f(z_{2j-2})\Big)\\
F^{2n+1}(z_0,w_0)&= (-\delta)^n\Big(-\delta\big( w_0+\sum_{j=1}^{n+1}(-\delta)^{-j}f(z_{2j-2})\big),z_0+\sum_{j=1}^n(-\delta)^{-j}f(z_{2j-1})\Big).
\end{align*}
For $n\in\N$ define the following holomorphic functions from $\C^2$ to $\hat{\C}$
\begin{align*}
\Delta_1^n(z_0,w_0)&:=  \sum_{j=1}^n (-\delta)^{-j}f(z_{2j-1})   \\
\Delta_2^n(z_0,w_0)&:=  \sum_{j=1}^n (-\delta)^{-j}f(z_{2j-2}) 
\end{align*}
With this notation the iterates of $F$ take the form 
\begin{align}
\label{eq:F1} F^{2n}(z_0,w_0)&= (-\delta)^n\Big((z_0+\Delta_1^n(z_0,w_0),w_0+\Delta_2^n(z_0,w_0)\Big)\\
\label{eq:F2}F^{2n+1}(z_0,w_0)&=(-\delta)^n \Big(-\delta w_0-\delta\Delta_2^{n+1}(z_0,w_0),z_0+\Delta_1^n(z_0,w_0)\Big).
\end{align}
Let
\begin{align*}
\Delta_1(z,w)&=\Delta_1^\infty(z,w):=\lim_{n\ra\infty} \Delta_1^n(z,w)\\
\Delta_2(z,w)&=\Delta_2^\infty(z,w):=\lim_{n\ra\infty} \Delta_2^n(z,w)\\
 \Delta(z,w) &=\max\left(\left| \Delta_1(z,w) \right|, \left|\Delta_2(z,w)\right|\right). 
\end{align*}
Notice that $\Delta_1,\Delta_2$  are holomorphic functions to $\hat{\C}$ on open sets on which  they are well defined.

We can deduce the following formal limits.
\begin{align} \label{eq:limit h1}
h_1(z,w)&:=\lim_{n\rightarrow \infty}\frac{z_{2n}}{w_{2n}}=\frac{z+\Delta_1(z,w)}{w+\Delta_2(z,w)}\\  \label{eq:limit h2}
h_2(z,w)&:= \lim_{n\rightarrow \infty}\frac{z_{2n+1}}{w_{2n+1}}=\frac{-\delta( w+\Delta_2(z,w))}{z+\Delta_1(z,w)}=-\frac{\delta}{h_1(z,w)}.
\end{align}
We have that $h_1,h_2$ are  holomorphic functions to $\hat{\C}$ on open sets on which $\Delta_1$ and $\Delta_2$ are holomorphic functions to $\hat{\C}$.
We will  show in Proposition~\ref{prop:h1 h2} that $h_1 \ne h_2$.

\subsection{Existence of Fatou components and rank of the limit functions}\label{sec:invariant open set}
In this section we   construct a forward invariant  open set   $W$ on which the even and the odd iterates converge, from which we deduce the existence of Fatou components. We then show that the limit functions have rank 1 on such Fatou components.
 
For  $A\subseteq \C^2$  and  $a,b \in \{ +,- \}$ define
\begin{equation}\label{eq:Aab}
A^{ab}:=A\cap \{(z,w)\in \C^2\:a\re(z)>0,\ b\re(w)>0 \}.
\end{equation}
If $A\cap(\{\Re z=0\}\cup \{\Re w=0\})=\emptyset$ then $A=\bigcup_{a,b \in \{+,-\}}A^{ab}.$

We start by defining a set on which we have control on the dynamics. Let  
\begin{align*}
\mathcal{S}&:=\{z\in \C:\,|\Im (z)|<|\Re(z)|\}  \subset\C         \\
S&:=\mathcal{S} \times \mathcal{S}\subset\C^2 
\end{align*}
\begin{lem}\label{lem:control of f}
Let $z\in \mathcal{S}$, then
$|f(z)|=|e^{-z^2}|< 1.$
\end{lem}
\begin{proof}
If $z\in \mathcal{S}$, then $|\arg(z)|< \frac{\pi}{4}$ and hence  $\re (z^2)>0$ from which we have $|e^{-z^2}|=e^{-\re z^2}<1$.
\end{proof}

\begin{lemma}[Orbits contained in $S$]\label{lem:forever in S}
For any   $P=(z_0,w_0)\in S^{ab}$ such that $F(P)\in S$ and $|\Re w_0|>\frac{1}{\delta}$ we have that  $F(P)\in S^{(-b)a}$. 

From now on assume that   $F^n(P)\in S $ for all $n\in\N$. Then 
\begin{align*}
F^{2n}(z_0,w_0)&\ra h_1(z_0,w_0)\\
F^{2n+1}(z_0,w_0)&\ra h_2(z_0,w_0).
\end{align*}
Fix  $\lambda>0$ and assume also  that   $|\re z_0|,|\re w_0|>\frac{1+\lambda}{\delta-1}$. Then
     \begin{align}
     \label{eq:Growth real odd}|\re z_{2n-1}|&=|\re w_{2n}|> |\re w_0|+n\lambda\\
 \label{eq:Growth real even}|\re z_{2n}|&=|\re w_{2n+1}|>|\re z_0|+n\lambda.
\end{align}
\end{lemma}

\begin{proof}
By hypothesis, $F(P)\in S$ hence $F(P)\in S^{\tilde{a}\tilde{b}}$ for some $\tilde{a},\tilde{b} \in \{+,-\}$.  Since  $\re w_1=\re z_0$ we have that  $\tilde b=a$. 
Moreover $\re z_1=-\delta \re w_0+\re(e^{-z_0^2}) $ and since $P\in S$, $|\re(e^{-z_0^2})|<1$ by Lemma~\ref{lem:control of f}.  Hence the sign of $\re z_1$ is  opposite to  the sign of $\re w_0$ provided  $|\re w_0|>\frac{1}{\delta}$, and  $\tilde a=-b$ as required. 

Assume from now on that  $F^n(P)\in S$ for all $n\in\N$.  It follows  that  $z_n\in \mathcal{S}$ for all $n\in\N$ and hence by Lemma~\ref{lem:control of f} $|f(z_n)|< 1$ for all $n\in\N$. Since $\delta>2$ this implies
\begin{equation}\label{eq: bounds on Delta in S}
\Delta(z_0,w_0)<\sum_{j=1}^\infty \delta ^{-j}<1 \text{  whenever  $F^n(z_0,w_0)\in S$  for all $n\in\N$}, 
\end{equation}
which implies convergence of the even and odd iterates of $F$ according to the expression in (\ref{eq:F1}),  (\ref{eq:F2}).

We now prove (\ref{eq:Growth real odd}), (\ref{eq:Growth real even}).
Using the expression of $F$ and since $P\in S $, by  Lemma~\ref{lem:control of f} we have $|\re z_1|\geq \delta |\re w_0|-|e^{-z_0^2}|\geq \delta |\re w_0|-1$ which is larger than $|\re w_0|+\lambda$ if $|\re w_0|>\frac{1+\lambda}{\delta -1}$. It follows that
 \begin{align}
\label{eq:growth1}|\re z_1|&> |\re w_0|+\lambda\\
|\re z_2|&>|\re w_1|+\lambda=|\re z_0|+\lambda, 
\end{align}
 where the claim for $z_{2}$ follows because $w_1=z_0$. The more general formula follows by induction, using that  $F^n(P)\in S $ for all $n\in\N$.
 \end{proof}

\begin{cor}\label{cor:points in sets contained in S}
Let $A\subset S$ be forward invariant. If $P=(z_0,w_0)\in A$ such that $|\re w_0|>\frac{1}{\delta}$ then Lemma~\ref{lem:forever in S} holds for $\lambda=1$, in particular, if  $P\in A^{ab}$ then $F(P)\in A^{(-b)a}$.
\end{cor} 

For   $R>0$ and  $0<k<1$ define the sets 
\begin{align*}
\mathcal{W}_{k,R}&:=\{z\in \C: |\im z|<k|\re z|, \ |\re z|>R\}\subset\C\\
W_{k,R_1,R_2}&:=\mathcal{W}_{k,R_1} \times \mathcal{W}_{k,R_2}\subset\C^2.
\end{align*}
Observe that $\mathcal{W}_{k,R} \subset \mathcal{S}$ and  that  $\mathcal{W}_{1,0}=\mathcal{S}$.
\begin{lem}\label{lem:invariance of single sets}
Let $n\in\N$, and let $(z_0,w_0)\in W_{k,R_1,R_2}$. 
 Let $0<k<\tilde{k} < 1$. If $R_2> \frac{2}{\delta(\tilde{k}-k)}$ then 
$$\left | \frac{\im z_1}{\re z_1}  \right |<
\tilde{k} \text{ and } \left | \frac{\im w_1}{\re w_1}  \right |<
{k}.$$
\end{lem}

\begin{proof}
Let $(z_0,w_0)\in W_{k,R_1,R_2}$. The claim for $w_1$ is immediate because $w_1=z_0$. Using the expression of $F$, the triangular inequality, the estimate in Lemma~\ref{lem:control of f} and the fact that $|\im  w_0|<k |\re w_0|$ we have
\begin{align*}
\Big|\frac{\im z_1}{\re z_1}\Big|&<\frac{\delta |\im w_0|+1}{\delta |\re w_0|-1}<\frac{k \delta |\re w_0 |+1}{\delta |\re w_0| -1}.
\end{align*}
Setting the resulting expression to be less than $\tilde k$ we get 
$|\re w_0|>\frac{1+\tilde{k}}{\delta (\tilde k-k)}$. Since $\tilde{k}< 1$, it is enough to take $|\re w_0|>\frac{2}{\delta (\tilde k-k)}$ as required.
\end{proof}
Let $k_n:=1-\frac{1}{n+2}$ and $R_n:=(\frac{\delta}{2})^{\frac{n}{2}} R_0$ for $R_0>2$ sufficiently large depending only on $\delta$ (see (\ref{eq:R0})). Let $R_{-1}=R_0$ and  set 
$$ W_n:=
\begin{cases}
W_{k_n,R_n,R_{n-1}}^{++} \text{ if $n=0\mod4$}\\
W_{k_n,R_n,R_{n-1}}^{-+}  \text{ if $n=1\mod4$}\\
W_{k_n,R_n,R_{n-1}}^{--}  \text{ if $n=2\mod4$}\\
W_{k_n,R_n,R_{n-1}}^{+-} \text{ if $n=3\mod4$}
\end{cases}.
$$
and define
$$W:=\bigcup_{n\in\N} W_n.$$

\begin{prop}[Invariance of $W$]\label{prop: invariance W}
The set  $W$ is open and $W\subset S$.
For any $n\in\N$ we have that $F(W_n)\subset W_{n+1}$, hence  $W$ is forward invariant.
 The set  $W$ consists  of four connected components $W^{ab}$ with $a,b\in\{+,-\}$ and $F(W^{ab})\subset W^{(-b) a}$.
\end{prop}

\begin{proof}
The fact that $W$ is open and $W\subset S$ follows from the definition.  
Fix $n\in\N$. Let $(z_0, w_0)\in W_n$ and let $(z_1,w_1)$ be its image. 
Since  $w_1=z_0$,  the signs of $\Re w_1,\Re z_0$ are the same, and we have that  $|\re w_1|=|\re z_0|>R_n$ and that
$$\left|\frac{\im w_1}{\re w_1}\right|=\left|\frac{\im z_0}{\re z_0}\right|<k_{n}<k_{n+1}.$$
Hence to show that $F(W_n)\subset W_{n+1}$ it is enough to see that $|\re z_1|>R_{n+1}$ and  that
$$
\left|\frac{\im z_1}{\re z_1}\right|<k_{n+1}.
$$
Let $\lambda_n:=R_{n+1}-R_{n-1}$. Since $P\in S$, by (\ref{eq:growth1}) we have that 
 $$
 |\re z_1|>|\re w_0|+\lambda_n > R_{n-1}+\lambda_n=R_{n+1}
 $$ 
 provided $R_{n-1}>\frac{1+\lambda_n}{\delta-1}$. Substituting the expression for $\lambda_n$ we get $R_{n+1}<\delta R_{n-1}-1$. Substituting the expression for $R_{n+1}$ and $R_{n-1}$ we get 
$$
\delta^{\frac{n+1}{2}}R_0>2^{\frac{n+1}{2}}
$$
which is satisfied because $\delta>2$, provided $R_0\ge1$. This gives $|\re z_1|>R_{n+1}.$

We now prove $\left|\frac{\im z_1}{\re z_1}\right|<k_{n+1}.$  By Lemma~\ref{lem:invariance of single sets}, it is enough to check that  $R_{n-1}>\frac{2}{\delta(k_{n+1}-k_n)}=\frac{2(n+2)(n+3)}{\delta}$, that is 
\begin{equation}
\label{eq:R0}
R_0>2^{\frac{n+1}{2}}\delta^{-\frac{n+1}{2}}(n+2)(n+3) \text{ for all $n\in\N$}.
\end{equation} 
Since the function on the right hand side is bounded in $n$ for any $\delta>2$ (in fact, it tends to $0$ as $n\ra\infty$), such $R_0$ exists and depends only on $\delta$. 

Finally, the set $W$ consists of $4$ connected components $W^{ab}$ by construction, since for any $(z,w)\in W$ we have $\Re z,\Re w\neq0$ and each $W^{ab}$ is connected. 
Since $W\subset S$ and we now know that $W$ is forward invariant, the orbits of points in $W$ are contained in $S$ hence Corollary~\ref{cor:points in sets contained in S} applies.  
\end{proof}
 
\begin{prop} [Existence of Fatou components]\label{prop:existence of FC}
On each $ W^{ab}$    we have that 
$$
F^{2n}\ra h_1, F^{2n+1}\ra h_2 \text{ uniformly on compact subsets of  $ W^{ab}$}.
$$
It follows that each $W^{ab}$  is contained in a  Fatou component that we denote by  $\Omega^{ab}$. 
\end{prop}
\begin{proof}
Since $W\subset S$ and is forward invariant by Proposition~\ref{prop: invariance W}, (\ref{eq: bounds on Delta in S}) holds hence   $F^{2n}$ and $F^{2n+1}$ converge uniformly on $W$ to $h_1,h_2$ respectively, hence $W$ is contained in the Fatou set. Since  each $W^{ab} $ is open and connected it is contained in a unique Fatou component that we denote by $\Omega^{ab}$. 
\end{proof}

We will see in Proposition \ref{prop:geometry of Omega} that in fact the components $\Omega^{ab}$ are all distinct and that the notation $\Omega^{ab}$ matches the definition of $A^{ab}$ given 
   in Section~\ref{sec:computing limit functions} for a general set $A$. 

\begin{prop}\label{prop:h1 h2}
Both  $h_1$ and $h_2$ have (generic) rank $1$ on $W$, and $h_1 \ne h_2$.
\end{prop}
\begin{proof}
Recall that $\Delta(z,w)<1$ on $W$ by (\ref{eq: bounds on Delta in S}). 
Since $h_i(W)\subset\ell_\infty$, $h_1$ and $h_2$ either have generic rank 1 or are constants.  Suppose by contradiction that $h_1=c$ is constant.
If $|c|\neq\infty$, then one has:
$$
|z_0|-\Delta(z_0,w_0)\leq \left|z_0+\Delta_1(z_0,w_0)\right|=|c|\left|w_0+\Delta_2(z_0,w_0))\right|\leq |c||w_0|+|c|\Delta(z_0,w_0),
$$
hence 
$$
|z_0|\leq |c||w_0|+(|c|+1),
$$
contradicting the fact that $(z_0,w_0)$ could be any point in   $W$, which is unbounded in the $z$ direction for any choice of $w$. 
If $c=\infty$,  we have $|w_0| \le 1$,  again   a contradiction. 
It follows  that $h_1 \ne h_2$. Indeed,    $h_1 \cdot h_2=-\delta$ is constant,  if we had $h_1=h_2$ it would follow  that $h_1^2$ (and hence $h_1$) would be constant as well, contradicting the argument above.
\end{proof}

\subsection{Construction of an absorbing set}\label{sec:absorbing}

Let $\Omega^{ab}$ with $a,b \in \{+,-\}$  be the Fatou components defined in  Proposition~\ref{prop:existence of FC} and let  $$\Omega:=\bigcup_{ab}\Omega^{ab}.$$  Since each $\Omega^{ab}$ is connected, $\Omega$ consists of at most 4 Fatou components. 
This section is devoted to find an absorbing set $W_I$ for $\Omega$ under $F$.  
 Its existence will be used in Section~\ref{sec:limit sets} to show that the Fatou components $\Omega^{ab}$ are all distinct and to describe both their limit sets and their geometric structure.  We use an argument based on harmonic functions  used also in \cite{FornaessShortCk}, \cite{henon1}, \cite{BSZ}.

 \begin{defn}[Absorbing sets] A set $A$ is \emph{absorbing} for an open  set $\Omega\supset A $ under a map  $F$   if for any compact $K\subset\Omega$ there exists $N>0$ such that 
 $$
 F^{n}(K)\subset A \text{ for all $n\geq N$.}
 $$
 \end{defn}
If $A$ is absorbing for $\Omega$, then $\Omega=\bigcup_{n}F^{-n}(A)$. 

 Fix $C\ge 1$ and let 
 $$
 I=I(C):=\{z\in \C\:\,|\im z|^2 < |\re z|^2-C^2\}=\{ z\in \C\,:\,\re (z^2) > C^2 \} \subset \mathcal{S}.
 $$
 Notice that if $z\in I$, then $|\re z|> C$. 
 
 Define 
 $$
W_I=W_{I}(C):=\Omega \cap\{(z,w) \in \C^2: F^n(z,w)\in I\times I \text{ for all $n\geq0$}\}.
 $$
\begin{prop} We have that $W_I^{ab}\neq\emptyset$ for all $a,b\in\{+,-\}$. For every $a,b\in\{+,-\}$,
$$
F(W_I^{ab})\subset W_I^{(-b)a}.
$$  
The sets $W_I^{++}\cup W_I^{--}$, $W_I^{-+}\cup W_I^{+-}$ are both forward invariant under $F^2$.  Moreover $F^{2n}$ and $F^{2n+1}$ are convergent on $W_I$. 
\end{prop}
\begin{proof}
Each $W_I^{ab}$ contains the set $\{(z,w)\in\C^2: a\re z> M,\ b\re w> M,\ \Im z=\Im w=0\}$ for $M$ sufficiently large. The set 
$W_I\subset S$  is forward invariant hence Corollary~\ref{cor:points in sets contained in S} applies. Convergence of even and odd iterates follow by (\ref{eq: bounds on Delta in S}.)
\end{proof}

It will turn out that $W_I$ is open as well (Proposition~\ref{prop:W_I open}). 

%
The rest of this  section is devoted to proving the following proposition.

\begin{prop}\label{prop:Absorbing WI}

The set $W_I$  is absorbing for $\Omega$ under $F$, that is, 
$$\Omega=\bigcup_n F^{-n} (W_{I})=:\AA_I.$$
\end{prop} 
Let
$$\mathcal{X}:=\{ (z,w)\in \Omega \,:\,h_1(z,w)=0,\infty \}.$$
Since $\mathcal X$ is an analytic set, being the union of the $0$-set and the $\infty$-set of a meromorphic function, it is locally a finite union of $1$-complex-dimensional varieties (see \cite{Chirka1989}).

Let $K$ be a compact subset of $\Omega \setminus\mathcal{X}$, hence $h_i(P)\ne 0,\infty$ for all $P\in K$, and $i=1,2$. Define
\begin{equation}
M:=\max_{K, i }|h_i|<\infty.
\end{equation}
Note that  $M>1$ because $h_2=-\frac{\delta}{h_1}$ and $\delta>1$. 
By Corollary 2.3 in \cite{BSZ} if $\epsilon>0$ is sufficiently small there exists a constant $c$ such that for every $(z_0,w_0)\in K$
 \begin{align}
 \label{eqtn:sandwich growth}
 |z_n|&\leq c(M+\epsilon)^n.\\
 \label{eqtn:sandwich growth2}
  |w_n|=|z_{n-1}|&\leq c(M+\epsilon)^{n-1}.
 \end{align}

The proof of Proposition~\ref{prop:Absorbing WI} relies on the following technical lemma. Recall that for $P=(z_0,w_0)$, we write $F^n(P)=(z_n,w_n)$.

\begin{lem}\label{lem:harmonic functions}
Define the sequence of harmonic functions $u_n$ from  $\Omega$ to $\R$ as 
\begin{equation}\label{eq:def un}
u_n(z_0,w_0):=\frac{-\re( z_n^2)}{n}.
\end{equation}
 Then 
\begin{enumerate}
\item  Let  $K\subset \Omega$ compact.  Then there exists $M=M(K) $ and $N\in \N$ such that $u_n\le \log M$ on $K$ for $n>N$;
\item $u_n\ra -\infty$  uniformly on compact subsets of  $W$; 
\item If  $P\in \Omega\setminus \AA_I$,  for every  $\epsilon>0$ there is a  subsequence $n_k\ra\infty$ such that $u_{n_k}(P)\ge -\epsilon$.
 \end{enumerate}
\end{lem}
\begin{lem} \label{lem:tan 2theta}
Let $z\in\C$, $k<1$.
If 
\begin{align}
\label{eq:k z}\left|\frac{\im z}{\re z}\right|&\leq k<1 \text{ then  }\\
\label{eq:k squared} \left|\frac{\im z^2}{\re z^2}\right|&\leq \frac{2k}{1-k^2}.
\end{align}
\end{lem}
\begin{proof}
Let $z=re^{i\theta}$ satisfying (\ref{eq:k z}); then $|\tan\theta|\leq k<1$. Hence since  $z^2=r^2e^{2i\theta}$,
$$
 \left|\frac{\im z^2}{\re z^2}\right|=|\tan(2\theta)|=\left|\frac{2\tan\theta}{1-\tan^2\theta}\right|\leq \frac{2k}{1-k^2}.
$$
\end{proof}
 
 The following fact is certainly  known, however we give a proof in   the Appendix. Given a set $A$, let $\mathring{A}$ denote its interior. 
 
\begin{prop}\label{prop: harmonic bounded outside analytic set}
Let $L$ be a compact set and $H$ be an analytic subset of dimension one of $\C^2$. For any compact $K$ s.t. $K\subset \mathring{L}$ there exists $\eta=\eta(K,L,H)$ such that for any $u$ harmonic defined in a neighborhood of $L$ and such that 
$$
u\leq\alpha<\infty \text{ on $L\setminus( \eta-$neighborhood of $H$) }
$$
we have 
$$
u\leq\alpha \text{ on $K$}
$$
\end{prop} 


\begin{proof}[Proof of Lemma~\ref{lem:harmonic functions}]\hfill
\begin{enumerate}
\item  
Let $K$ be a compact subset of $\Omega$.  Let $\eta$ as obtained by applying  Proposition~\ref{prop: harmonic bounded outside analytic set} to a slightly larger compact set $L\subset\Omega$ and to the analytic set $\XX$. Let $U_\eta(\XX)$ be an $\eta$-neighborhood of $\XX$.
In view of Proposition~\ref{prop: harmonic bounded outside analytic set} it is enough to prove that there exists $N\in \N$ such that $u_n\le \log M$  for $n>N$ and for some $M$ on the set
$$
\tilde K:=K\setminus U_\eta(\XX)
$$
which is  a compact subset of $\Omega\setminus \XX$. Hence it is enough to prove the claim for any $K$ compact subset of $\Omega\setminus \XX$.

Fix $\epsilon>0$ suffciently small and let $M, c$ be as in (\ref{eqtn:sandwich growth}) and (\ref{eqtn:sandwich growth2}) for $K$. Suppose that there exists a subsequence $(n_j)$ and points $(z,w)=(z(j), w(j))\in K$  such that
$$-\frac{\re({z}_{n_j}^2)}{n_j}>\beta$$
for some $\beta$. We will show that $\beta\leq M$. 

Using (\ref{eqtn:sandwich growth}) and (\ref{eqtn:sandwich growth2}) we have that 
\begin{align*}
 c(M+\epsilon)^{n_j+1}&\geq   |z_{n_{j+1}}|=|e^{-z_{n_j}^2}-\delta w_{n_j}| \ge |e^{-z_{n_j}^2}|-\delta|w_{n_j}|\ge \\
    &\geq e^{-\re(z_{n_j}^2)}-\delta c(M+\epsilon)^{n_j-1} \geq  e^{\beta n_j}-\delta c (M+\epsilon)^{n_{j}-1}.
    \end{align*}
 Hence, using  $M>1$ and $\epsilon >0$ sufficiently small,
$$
    e^{\beta n_j}\leq \delta c (M+\epsilon)^{n_{j}-1}+ c(M+\epsilon)^{n_j+1}\leq c(\delta +1)(M+\epsilon)^{n_j+1}.
$$

Then
$$\beta \leq \frac{\log\big(c(\delta +1)\big)}{n_j}+\frac{n_j+1}{n_j}\,\, \log(M+\epsilon)\ra \log M$$
as $n_j\longrightarrow \infty$ and $\epsilon \longrightarrow 0$.

\item 
It is enough to show  that $u_n(z_0, w_0)\ra-\infty$ for any point $(z_0,w_0)\in W$ and it will follow for any compact subset of $W$.  Since $W$ is forward invariant, $F^n(z_0,w_0)\subset W\subset S$ for all $n\in \N$ and $\Delta(z_0,w_0)<1$ by (\ref{eq: bounds on Delta in S}). Using the explicit expression for iterates of $F$ given by (\ref{eq:F1}), (\ref{eq:F2}) we have 
$$
|z_n^2|=|z_n|^2=
 \begin{cases} \delta^{n}|z_0+\Delta_1^{n/2}(z_0,w_0)|^{2}\geq \delta^{n}|z_0-1|^2  \ \ \  \text{ if n even};\\
 \delta^{(n+1)}|w_0+\Delta_2^{(n+1)/2}(z_0,w_0)|^2\geq \delta^{n+1}|w_0-1|^2  \ \ \  \text{ if n odd}.
 \end{cases}
 $$
 
 In both cases, since  $|z_0|, |w_0|>R_0>2$ we obtain $|z_n^2|\geq  \delta^{n} $. Since $W=\bigcup_j W_j$ as defined in Section~\ref{sec:invariant open set}, $(z_0,w_0)\in W_j$ for some $j$, hence by Proposition~\ref{prop: invariance W}, 
$$
F^n(z_0,w_0)\in W_{j+n} \text{ for all $n\in\N$, }
$$ 

hence $\left|\frac{\im z_n}{\re z_n}\right|\leq k_{j+n}<1$   and by Lemma~\ref{lem:tan 2theta} we obtain
$$
\left|\frac{\im z_n^2}{\re z_n^2}\right|\leq \frac{2k_{j+n}}{1-k_{j+n}^2}=:\alpha_n\sim n \text{ as $n\ra\infty$},
$$
where the estimate   $\alpha_n\sim n$ as $n\ra\infty$ is computed using the explicit expression for   $k_{j+n}$.
It follows that 
  $$\delta^n\leq |z_n^2|=\sqrt{(\re (z_n^2))^2+(\im (z_n^2))^2}\leq \Re (z_n^2)\sqrt{1+\alpha_n^2}$$
   hence
  $\re (z_n^2)\geq \frac{\delta^{n}}{\sqrt{1+\alpha_n^2}}\sim\frac{\delta^{n}}{n}\geq \delta^{n/2}$ for $n$ large. Finally
  
  $$
  u_{n}(z_0,w_0)=-\frac{\Re(z_n^2)}{n}\leq -\frac{\delta^{n/2}}{n}\ra-\infty \text{ as $n\ra\infty$}\
  $$

\item 
 Suppose by contradiction that there exists $P=(z_0,w_0)\in \Omega\setminus\AA_I$,  $\epsilon>0$ and $N\in \mathbb{N}$ such that 
$$
u_n(z_0,w_0)=\frac{-\re z_n^2}{n} < - \epsilon \text{ for all $n\ge N$.}
$$ 
Hence there exists $N'>N$ depending on $\epsilon, C$ (where $C$ is the constant used to define $W_I$)  such that 
$$
\re(z_n^2)> \epsilon n> C^2 \text{ for all $n\geq N'$}.
$$

Since $w_n=z_{n-1}$ and since $P\in \Omega$ for hypothesis, we have that  $F^n(P)\in I\times I$ for all $n\ge N'$  hence  $P\in F^{-n}(W_I)\subset \AA_I$, a contradiction.
\end{enumerate}
\end{proof}

\begin{lem}[Good holomorphic disks]\label{lem:Good holo disks}
Let $P\in \Omega$, $W$ as  before. 
Then there exists $\phi:\overline{\D}\ra \Omega$ holomorphic such that
\begin{itemize}
\item $\phi(0)=P$
\item $\phi(\D)\Subset\Omega$ and $\partial \phi(\D)$ is analytic
\item The one-dimensional Lebesgue measure of $\partial \phi(\D)\cap W$ is greater than $0$. 
\end{itemize}
\end{lem}
\begin{proof}
Since $W$ is open it is enough to have  $\phi(\D)\cap W\neq\emptyset$ to ensure that the one-dimensional Lebesgue measure of $\partial \phi(\D)\cap W$ is greater than $0$. Let $a,b\in \{+,-\}$ such that $P\in \Omega^{ab}$. Since $W^{ab}\neq\emptyset$ for all $a,b\in\{+,-\}$ there exists $Q\in W^{ab}$. Since $\Omega^{ab} $ is connected and open there exists  a  simple real analytic curve passing through $P$ and $Q$ in $\Omega^{ab}$. Complexifying  this  curve we obtain a holomorphic  disc passing through $P$ that we can write as $\phi(\D)$ for some $\phi$ holomorphic defined in a neighborhood of $\D$.  Up to precomposing $\phi$ with a Moebius transformation we can assume that $P=\phi(0)$. 
\end{proof}

In our proof, we are going to use the mean value property for the harmonic functions $u_N$.

\begin{lem} [Mean value property for holomorphic disks]\label{MVP} 
Let $\D\subset\C$ be the open unit disk and $\phi:\overline{\D}\ra \Omega$ be a  holomorphic map. Let $u$ be harmonic on the holomorphic open disk $D=\phi({\D})$ and continuous up to the boundary of $D$.  Let $P_0:=\phi(0)$. Then 
$$
u(P_0)=\frac1{2\pi}\int_{\partial \D}u(\zeta)|\phi'(\zeta)|^{-1}d\zeta
$$
\end{lem}
\begin{proof}
    Consider the function  $u\circ \phi:\overline\D\to\R$. First, note that it is harmonic on $\D$ and continuous up to the boundary. Indeed if $u:D\to\R$ is $\mathcal{C}^2$-smooth, then we can explicitly compute its Laplacian 
    $$\nabla^2(u\circ\phi)=\nabla^2(u)|\phi'|^2=0\,$$
    while if $u$ is not $\mathcal{C}^2$-smooth, the result follows by approximating $u$ with harmonic smooth functions.
    
    Hence for $u\circ \phi$ the classical Mean Value Property holds.  By computing $u(P_0)$ we get
\begin{equation}\label{eq: mean value generalized}
u(P_0)\,=\,u(\phi(0))\,=\,\frac1{2\pi}\int_{\partial \D}u(\phi(\eta))\,d\eta\,=\,\frac1{2\pi}\int_{\partial \D}u(\zeta)\,|\phi'(\zeta)|^{-1}\,d\zeta.
\end{equation}
\end{proof}

\begin{proof}[Proof of Proposition~\ref{prop:Absorbing WI}]
 
Let $P\in \Omega \setminus \AA_I$ and  $D:=\phi(\D) $ where $\phi$ is given by Lemma~\ref{lem:Good holo disks}.  Let $\mu$ be the pushforward  under  $\phi$ of the one-dimensional Lebesgue measure on $\partial \D$. Let $K $ be a compact subset of $W$ such that $\mu (K\cap \partial D)>0$.  

Let $\mu_{\text{good}}=\mu(\partial D\cap K)>0$ and $\mu_{\text{bad}}=\mu(\partial D\cap (\Omega \setminus K))$. Since $\Omega$ contains $D$,  $\partial D=(\partial D\cap K )\cup (\partial D\cap (\Omega\setminus K))$, and since $K$ is compact and $\Omega$ is  open, the sets in question are measurable.  

By Lemma~\ref{lem:harmonic functions} for any given $\MM>0$ there exists $N$ such that $u_N\leq -\MM$ on $K$, $u_N(P)\geq -\epsilon$ for  some $\epsilon>0$  since $P\in \Omega\setminus \AA_I$, and $u_N\leq \log M$  on  $\ov{D}$ (with $M=M(\ov{D})$). By the Mean value property (\ref{eq: mean value generalized}) for $u_N$ we have 
\begin{align*}
-\epsilon\leq u_N(P)&=\frac{1}{2\pi}\int_{\partial D} u_N(\zeta)|\phi'(\zeta)|d\zeta=\frac{1}{2 \pi}\int_{\partial D\cap K} u_N(\zeta)|\phi'(\zeta)|d\zeta+\frac{1}{2\pi}\int_{ \partial D\cap (\Omega \setminus K) } u_N(\zeta)|\phi'(\zeta)|d\zeta\leq\\
&\leq \frac{1}{2\pi} \left(-\MM \mu_{\text{good}}+ \log M\mu_{\text{bad}}\right)\cdot \sup_{\partial\D}|\phi'|^{-1}. 
\end{align*}
Since $\MM$ is arbitrarily large, this gives  a contradiction.  
\end{proof}
\begin{prop}\label{prop:W_I open}
The set  $W_I$ is open.
\end{prop}
\begin{proof}
Let $P \in W_I$. We want to find $V\subset W_I$ neighborhood of $P$. Since $W_I\subset \Omega\cap (I\times I)$ which is open there is a neighborhood $U$ of $P$ which is compactly contained in $\Omega\cap (I\times I)$.
Since $W_I$ is absorbing for $\Omega$ under  $F$ there exists $N>0$ such that 
\begin{equation}\label{eq:absorbed}
F^n(\ov{U})\subset W_I \text{ for all $n\geq N$}.
\end{equation}
As usual let us define $P_j:=F^j(P)$; by definition of $W_I$,  $P_j\subset I\times I$ for all $j\ge0$, which is an open set. Hence for each $j\geq0$ there is a neighborhood $U_j\subset I\times I$ of $P_j$. So up to making  the $U_j$ smaller, we can assume that $U_j\subset F^j(U)$.

Let 
$$ 
V:=\bigcap_{j=0}^N F^{-j}(U_j)\subset U. 
$$

The set $V$ is open since it is a finite intersection of open sets. We only need to check that $V\subset W_I$, or equivalently, that $F^j(V)\subset I\times I $ for all $j\geq0$. For $j\leq  N-1$, this is true by definition, since $F^j(V)\subset U_j\subset I\times I$. For $j\geq N$, this is true by  (\ref{eq:absorbed}).  Since $P\in V$ by construction, $V$ is a neighborhood of $P$ in $W_I$ as required.
\end{proof}

\subsection{Geometric structure of $\Omega$}\label{sect: limt sets}
 In this section we show  that $\Omega$ is the union of four disjoint Fatou components $\Omega^{ab}$, $a,b\in\{+,-\}$, each of which is biholomorphic to $\H\times\H$.

We first  show conjugacy of $F$ to its linear part on $\Omega$, and estimate the distance between the conjugacy and the identity map. 

 \begin{prop}[Conjugacy]\label{prop:conjugacy} $F$ is conjugate to the linear map $L(z,w)=(-\delta w,z)$ on $\Omega$. If  $P$ is such that $F^n(P)\in S$ for all $n\in\N$, then $\|(\phi - Id) (P)\|<\sqrt{2}$.  Finally,  $\phi(\Omega)\subset S$.
 \end{prop}
 \begin{proof}
 We first show that $F$ is conjugate to $L$ on $W_I$.

 For $n\in\N$ let $\phi_n:\C^2\ra\C^2$ be the automorphisms defined as
  $$
\phi_n:=L^{-n}\circ F^n. 
 $$
If we  show that the  $\phi_n$ converge to a map $\phi:\C^2\ra\C^2$   uniformly on  $W_I$ we obtain that   $\phi$ satisfy the functional equation $\phi=L^{-1}\circ\phi\circ F$ and hence is a conjugacy between $F$ and $L$.
 
 Computing $L^{-n}$ and using the explicit expressions for the iterates of $F$ we obtain
 \begin{align}
 \phi_{2n}(z,w)&=\Big(z+\Delta_1^n(z,w),w+\Delta_2^n(z,w)\Big),\\
\phi_{2n+1}(z,w)&=\Big(z+\Delta_1^n(z,w),w+\Delta_2^{n+1}(z,w)\Big).
\end{align}
Both have the same  formal  limit 
 $$
 \phi (z,w)=\Big(z+\Delta_1(z,w),w+\Delta_2(z,w)\Big).
 $$
If   $P=(z,w) \in W_I$,  then $F^n(P)=(z_n,w_n)\subset I\times I\subset S$ for all $j$, hence, by (\ref{eq: bounds on Delta in S}),  we have that 
$\Delta(z,w)<1$;  in particular, $\Delta_1(z,w)$ and $\Delta_2(z,w)$ are convergent. Hence $\phi$
 is  a holomorphic map from  $W_I$ to $\phi(W_I)$  ($W_I$ is open by Proposition~\ref{prop:W_I open}).
 Moreover, 
\begin{equation}\label{eq: estimates on conj}
 \Big \|(\phi - \Id) (z,w)\Big\|=\Big \|\Big(\Delta_1(z,w),\Delta_2(z,w)\Big)\Big\|<\sqrt{2}\Delta(z,w)<\sqrt{2}.
\end{equation}
 It follows that $\phi$ is open  because $W_I$ is an unbounded set, hence if $\phi$ had rank 0 or 1, $\left \|(\phi - \Id)\right \|$ could not be bounded on $W_I$.  Hence the map $\phi$ is injective by Hurwitz Theorem  (see \cite{Krantz}, Exercise 3 on page 310) because the maps $\phi_n$ are injective  and their limit has rank 2. It follows that $\phi$ is a biholomorphism between $W_I$ and  $\phi(W_I)$.
  
To extend $\phi$ to all of $\Omega$ recall that $W_I$ is absorbing for $\Omega$. So if $P\in\Omega$, we have that  $F^k(P)\in W_I$ for some $k\in\N$, hence we can define   $\phi(P)= L^{-k}\circ \phi\circ F^k(P)$. Since $F$ is an automorphism, $\phi$ extends  as a biholomorphism  to $\Omega$. 
   
   It remains to show that $\phi(\Omega)\subseteq S$. By (\ref{eq: estimates on conj}) we have that  $\phi(W_I)$ is contained in a $\sqrt{2}$ neighborhood $U$ of $W_I$. 
Suppose by contradiction that there exists  $Q=(z,w)\in \phi(W_I)\setminus S$.  Since $W_I$ is forward invariant under $F$ and  $\phi$ is a conjugacy we have that  $\phi(W_I)$ is forward invariant under $L$. Up to considering $L(Q)$ if necessary,
and since $\theta$ is such that $r e^{i\theta}\notin S$,
 we can assume that $z=re^{i \theta}\notin \SS$.  By forward invariance  $L^{2n}(Q)=((-\delta^{n}) r e^{i\theta},(-\delta)^{n}w) \in \phi(W_I)$. 

Since $(-\delta)^{n} r $ tends to infinity, the distance of $L^{2n}(Q)$ from  the boundary of $S$ tends to infinity, hence so does the distance of  $L^{2n}(Q)$ from $W_I\subset S$, contradicting   $\phi(W_I)\subset U$. Hence $\phi(W_I)\subset S$. Since  $W_I $ is an absorbing set for $\Omega$ under $F$, $\phi\circ F=L\circ\phi$, and $\phi(W_I)$ is completely invariant under $L$,   we have that
\begin{equation}\label{eq:image of Omega under conjugacy}
\phi(\Omega)= \phi(\bigcup_{n\geq0} F^{-n}(W_I))=\bigcup_{n\geq0} L^{-n}(\phi(W_I))\subset \phi(W_I) \subset S.
\end{equation}
\end{proof}

We are now able to understand the geometric structure of $\Omega$.  
\begin{prop}\label{prop:geometry of Omega}
$\Omega$ consists of four distinct connected components, each of which is  biholomorphic to $\H\times\H$, and which form a cycle of period 4.
\end{prop}
We recall the following simple  topological lemma. Here $\partial$ denotes the topological boundary.
\begin{lem}\label{lem:topological lemma}
Let $A,B\subset\C^n$ be open, $A$ connected. If $A\cap B\neq\emptyset$ and $\partial B\cap A=\emptyset$ then $A\subseteq B$.
\end{lem}
\begin{proof}
Since $A\cap \partial B=\emptyset$ we can write 
$$
A=(A\cap B)\cup (A\setminus \ov{B}).
$$
Both $A\cap B$ and $A\setminus \ov{B}$ are open and  $A\cap B\neq\emptyset$ by assumption, so  since $A$ is connected, $A\setminus \ov{B}=\emptyset$. 
\end{proof}
Recall also that if  a set $A$ is invariant under a map $F$, by continuity of the latter we have $F(\ov{A})\subset \ov{A}$. 
The following lemma is also known.

\begin{lem}\label{lem:Fatou components for automorphisms}
Let $\Omega_1,\Omega_2$ be two Fatou components for an automorphism  $F$ of  $\C^2$. Then if $F(\Omega_1)\cap \Omega_2\neq\emptyset$, $F(\Omega_1)=\Omega_2$.
 \end{lem}
 
\begin{proof}
We have that $F(\Omega_1)\subset \Omega_2$, indeed otherwise, $F(\Omega_1)$ would intersect the boundary of $\Omega_2$ which is contained in the forward Julia set, and this is impossible because the Fatou set is completely invariant.
On the other hand suppose for a contradiction that there is $P\in \Omega_2\setminus F(\Omega_1)$. Then since $F(\Omega_1)\cap \Omega_2\neq\emptyset$ and both $\Omega_1, F(\Omega_1)$ are connected there exists $Q\in\Omega_2\cap F(\partial\Omega_1)$, which  is impossible because $\partial \Omega_1$  is contained in the forward Julia set which is forward invariant. 
\end{proof}
Observe that we could not simply use the same argument applied to $F^{-1}$, since the Fatou components for $F$ and $F^{-1}$ are, in general, different sets.

\begin{proof}[Proof of Proposition~\ref{prop:geometry of Omega}]
We prove the claim by showing that $\Omega$ is biholomorphic to $S$. Since $S$ has four connected components $S^{ab}$ each of which is biholomorphic to  $\H\times\H$, the same holds for $\Omega$. Since by definition $\Omega$ had at most four connected components $\Omega^{ab}$ with $a,b\in\{+,-\}$, these are exactly the connected components of $\Omega$.  

Recall the definition of the set $W\subset S$ from Section~\ref{sec:invariant open set} and  recall that it is forward invariant. We first show that 
\begin{equation}\label{eq:image under conjugacy intersects}
\phi(W^{ab})\cap S^{ab}\neq\emptyset \text { for every $a,b\in\{+,-\}$;}
\end{equation} 
Indeed, by construction every  $W^{ab}$ contains points  $(z,w)$ with $|z|,|w|>M>\sqrt{2}$. For such points, by (\ref{eq: estimates on conj}) we have that $\|\phi(z,w)-(z,w)\|<\sqrt{2}$, but since $W^{ab}\subset S^{ab}$ by Proposition~\ref{prop:conjugacy},  the Euclidean distance of $(z,w)$ from $S^{\tilde{a}, \tilde{b}}$ is  larger than $M$ if $(a,b)\neq (\tilde{a}, \tilde{b})$ (indeed, if you take a point in one sector, no point  in a different sector is closer than  the origin).

We now show that $\phi(\Omega)=S$. By (\ref{eq:image of Omega under conjugacy})  we know  that $\phi(\Omega)\subset S$. 
Notice that $S$ can be written as
$$
S=\{
(r_1 e^{i\theta_1},r_2 e^{i\theta_2})\in\C^2: \text{$r_1,r_2>0$, and for each $i=1,2$ either  $|\theta_i|<\frac{\pi}{4}$   or $|\theta_i-\pi|<\frac{\pi}{4}$}
\}.
$$
Fix $\epsilon>\sqrt{2}$ and let  $U$ be an  $\epsilon$-neighborhood of $\partial W$. Since  $\|\phi-\Id\|<\sqrt{2}$ on ${W}$,  by continuity of $\phi$ we have that $\|\phi-\Id\|\leq \sqrt{2}$ on ${\partial W}$, hence $\phi(\partial W) \Subset U$. Let 
$a,b\in\{+,-\}$ and let  $B=\phi(W^{ab})$, $A=W^{ab}\setminus\ov{ U}$. By (\ref{eq:image under conjugacy intersects}) we have that $A\cap B\neq\emptyset$, 
and by the previous argument, $\partial B=\partial \phi(W^{ab})\subset \phi(\partial W^{ab})\Subset U$ does not intersect $A$.

 Hence, applying   Lemma~\ref{lem:topological lemma} to each possible choice of $a,b\in\{+,-\}$ we obtain  that 
$$
\phi(\Omega)\supset \phi(W)\supset (W\setminus U).$$
   Fix $\alpha<\frac{\pi}{4}$. By definition of $W$ there exists $R=R(\alpha)$ such that  $W\setminus U$ contains the set 
$$
W\setminus U\supset X_{\alpha,R}:=\{(r_1 e^{i\theta_1}, r_2 e^{i\theta_2}):\text{$r_1,r_2>R$ and for each $i=1,2$ either $|\theta_i|<\alpha$ or $|\theta_i-\pi|<\alpha$}\}.
$$ 
 Hence $\phi(\Omega)\supset  \phi(W)\supset (W\setminus U)\supset X_{\alpha,R}$. 
 By the explicit form of $L$,  $\bigcup_{j\geq0}L^{-j}X_{\alpha,R}=X_{\alpha,0} $.
   Hence   by backward invariance of $\phi(\Omega)$ under $L$ we have that 
    $$
 \phi(\Omega)\supset \bigcup_{j\geq0}L^{-j}X_{\alpha,R}=X_{\alpha,0} \text{ for every $\alpha<\frac{\pi}{4}$}.
 $$
 It follows that
   $$
\phi(\Omega)\supset\bigcup_{\alpha<\frac{\pi}{4}}    X_{\alpha,0} =S.
   $$ 

It remains to show that the Fatou components $\Omega^{ab}$ with $a,b\in\{+,-\}$ form a cycle of period four, more precisely, that 
\begin{equation}\label{eq:cycle on Omega}
F(\Omega^{ab})=\Omega^{(-b)a} \text{ for all $a,b\in\{+,-\}$.}
\end{equation}   
  By definition $\Omega^{ab}\supset W^{ab}$ and by Lemma~\ref{lem:forever in S}, $F(W^{ab})\cap W^{(-b)a}\neq\emptyset$. Hence  $F(\Omega^{ab})\cap \Omega^{(-b)a}\neq\emptyset$. By Lemma~\ref{lem:Fatou components for automorphisms},  $F(\Omega^{ab})= \Omega^{(-b)a}$.
  
\end{proof}
 
 \subsection{Limit sets on $\Omega^{ab}$}\label{sec:limit sets}
 Let $\H$ and $-\H$ denote the right and left half plane respectively.
In this section we show that 
 \begin{prop}[Limit set for $\Omega$]\label{prop:limit set for Omega}
 We have that
\begin{align*}
& h_1(\Omega^{ab})={\H} &\text{ and }&  & h_2(\Omega^{ab})=-{\H}  &\text{ if $a=b$ } \\
& h_1(\Omega^{ab})=-{\H} &\text{ and }&  & h_2(\Omega^{ab})={\H}  &\text{ if $a\neq b$}.
\end{align*}
 \end{prop}

 Let $W$ be as defined in Section~\ref{sec:invariant open set} and $W_I$ as defined in Section~\ref{sec:absorbing}. Since both are forward invariant and contained in $S$ we have that, for any $a,b\in\{+,-\}$,  $F(W^{ab})\subset W^{(-b)a}$ and $F(W_I^{ab})\subset W_I^{(-b)a}$. Compare with  Lemma~\ref{lem:forever in S} and Corollary~\ref{cor:points in sets contained in S}.

We first study the  image of $W_I^{ab}$  under $h_1,h_2$.
  \begin{lemma}
 \label{lem:limit set for WI}
\begin{align*}
&h_1(W_I^{ab})\subset {\H} &\text{ and }&  &h_2(W_I^{ab})\subset {-\H}  &\text{ if $a=b$ } \\
 &h_1(W_I^{ab})\subset {-\H} &\text{ and }& &h_2(W_I^{ab})\subset {\H}  &\text{ if $a\neq b$ } 
\end{align*}
 \end{lemma}
  \begin{proof}
Recall that $h_1(z_0,w_0)=\lim_{n\ra\infty}\frac{z_{2n}}{w_{2n}}$. Let $I_+:=\left(-\frac{\pi}{4},\frac{\pi}{4}\right)$ and $I_-:=\left(\frac{3}{4}\pi,\frac{5}{4}\pi\right)$.

For $a,b\in\{+,-\}$ and  $(z,w)\in W_I^{ab}$, then $\arg (z)\in I_a$ and $\arg( w)\in I_b$. Hence  $\arg\left( \frac{z}{w}\right)\in \left(-\frac{\pi}{2},\frac{\pi}{2} \right)$ if $a=b$, and $\arg\left( \frac{z}{w}\right)\in \left( \frac{\pi}{2},\frac{3}{2}\pi\right) $ if $a\neq b$. 
Since $F^2(W_I^{++}\cup W_I^{--})\subset W_I^{++}\cup W_I^{--}$, If $(z,w)\in W_I^{++}\cup W_I^{--}$  then all of its even iterates $(z_{2n}, w_{2n})\in W_I^{++}\cup W_I^{--}$, hence by taking the limit  $h_1 (W_I^{++}), h_1(W_I^{--})\subset \ov{\H}$. Similarly if  $(z,w)\in W_I^{+-}\cup W_I^{-+}$  then all of its even iterates $(z_{2n}, w_{2n})\in W_I^{+-}\cup W_I^{-+}$, hence  $h_1 (W_I^{+-}), h_1(W_I^{-+})\subset \ov{-\H}$. The analogous results for $h_2$ hold by observing that $h_2=\frac{-\delta}{h_1}$.  Since $W_I$ is open by Proposition~\ref{prop:W_I open},  its image under a   holomorphic  map of maximal rank is open, hence we can replace $\overline{\H},\overline{-\H}$ by $\H,-\H$.
  \end{proof}

 \begin{lem}
 \label{lem:limit set for W}
\begin{align*}
&\H\subset h_1(W^{ab})  &\text{ and }&  &-\H\subset h_2(W^{ab})  &\text{ if $a=b$ } \\
 -&\H\subset h_1(W^{ab}) &\text{ and }& &\H\subset h_2(W^{ab}) &\text{ if $a\neq b$}.
\end{align*}
 \end{lem}

  Before proving Lemma~\ref{lem:limit set for W} let us see how Lemmas~\ref{lem:limit set for WI} and \ref{lem:limit set for W} imply Proposition~\ref{prop:limit set for Omega}.
  
  \begin{proof}[Proof of Proposition~\ref{prop:limit set for Omega}]
  We prove the claims for $h_1$; for $h_2=\frac{-\delta}{h_1}$, it follows by symmetry. 
  
Clearly $h_1(\Omega^{ab})\supset h_1(W^{ab})$  for any $a,b\in\{+,-\}$ since $\Omega^{ab}\supset W^{ab}$. So  in view of Lemma~\ref{lem:limit set for W},   $h_1(\Omega^{ab})\supset -\H$ or  $h_1(\Omega^{ab})\supset \H$  depending on whether $a=b$.
 
We now  consider limit sets for $\Omega^{++}$ and $\Omega^{--}$; the other cases are analogous.
By (\ref{eq:cycle on Omega}),  
$$
F^{2}(\Omega^{++})\subset \Omega^{--} \text{ and } F^{2}(\Omega^{--})\subset \Omega^{++}.
$$
It follows that for any $n>0$, 
 $$
F^{2n}(\Omega^{++}\cup\Omega^{--})\subset\Omega^{++}\cup \Omega^{--}.
$$
In view of this, and since $W_I$ is absorbing for $\Omega$,  we have that  for any $P\in(\Omega^{++}\cup\Omega^{--} )$
 $$
F^{2n}(P)\subset W_I\cap(\Omega^{++}\cup \Omega^{--})=W_I^{++}\cup W_I^{--} \text{ for any $n$ large enough}. 
$$
Hence $h_1(P)\in h_1(W_I^{++}\cup W_I^{--} )\subset{\H}$ for every $P\in(\Omega^{++}\cup\Omega^{--} )$, hence $h_1(\Omega^{++}\cup\Omega^{--} )\subset{\H}$. 
It follows that  $h_1(\Omega^{++})=h_1(\Omega^{--})= \H$.

\end{proof}

We devote the rest of this section to proving Lemma~\ref{lem:limit set for W}.  We first give a version of Rouch\'e's Theorem  which relies on one of the many versions of Rouch\'e's Theorem existing  in one variable (compare with Theorem 3.4 in \cite{BSZ}; we will use it with the spherical instead  of the Euclidean metric).  
 This is certainly known to experts in the field but we are not aware of a reference. In this section  $\partial $ denotes the topological boundary, and  $\dist_{\operatorname{spher}}$ denotes the spherical  distance.

\begin{thm}[Rouch\'e' s Theorem in $\C^2$]\label{thm:Rouche 2D}
Let $B\subset \C^2$ be a polydisk, $F,G$ be holomorphic maps defined in a neighborhood of $\ov{B}$ which take values in $\hat{\C}$.  Let $c\in G(B)$, let $\epsilon= \dist_{\operatorname{spher}}(c, G(\partial B))>0$ and assume 
 $$
 \dist_{\operatorname{spher}}( F,G)<\epsilon \text{ on $\partial B$}.   
 $$
 Then $c\in F(B)$. 
\end{thm}
Notice that the assumptions imply that $F,G$ have generic rank 1: They cannot have rank 2 because the target is $\hat{\C}$, and $G$ cannot be constant otherwise there could not be $c\in G(B)$ with positive distance from $G(\partial B)$. One can check that $F$ cannot be constant either.
\begin{proof}
Let $D$ be a horizontal disk passing through a point $P_c\in G^{-1}(c)\cap B$, such that $\partial D\subset \partial B$. Let $g:=G|_{D}$, $f:=F_{D}$. They are holomorphic in a slightly larger horizontal disk. 
Notice that $\dist_{\operatorname{spher}}(g,f)<\epsilon$ on  $\partial D$, and that $\dist_{\operatorname{spher}}(G(\partial D),c)\geq\epsilon$ because $\partial D\subset \partial B$. 
By Rouch\'e's Theorem in one variable,  $c\in f(D)\subset F(B)$ as required.
\end{proof}
\begin{rem} Unless $P_c$ is an isolated point in $G^{-1}(c)$ we obtain a curve of points in $F^{-1}(c)$. Indeed, the proof gives a point in $F^{-1}(c)$ for any Euclidean disk passing through points in $G^{-1}(c)$, for example, a family of disks passing through $P_c$ along different complex directions.  The points obtained for $F^{-1}(c)$ are  distinct unless they always coincide with $P_c$. On the other hand,  if $P_c$ is an isolated point in $G^{-1}(c)$ then $P_c\in F^{-1}(c)$ is also isolated in $F^{-1}(c)$. Indeed otherwise we could reverse the role of $F$ and $G$  and obtain one point in $G^{-1}(c)$ for any Euclidean disk passing through any point in $F^{-1}(c)$ and obtain a curve of points in $G^{-1}(c)$. The proof as it is works when $B$ is any $\C$-convex set instead of a polydisk, and can certainly be generalized further.
\end{rem}

\begin{proof}[Proof of  Lemma~\ref{lem:limit set for W}]
We  show $\H\subset h_1(W_I^{++})$. The other cases are analogous. Recall that orbits of points in $W$ are contained in $S$, hence (\ref{eq: bounds on Delta in S}) holds. Since
$$
\frac{z_{2n}}{w_{2n}}=\frac{z_0+\Delta_1^n(z_0,w_0)}{w_0+\Delta_2^n(z_0,w_0)},
$$
dividing the numerator and the denominator by $w_0$ and using   $\frac{1}{1+x}=\sum_{j=0}^\infty( -x)^j$ for $|x|<1$ we have that 
\begin{equation}\label{eq: distance h0 h1}
\frac{z_{2n}}{w_{2n}}-\frac{z_{0}}{w_{0}}= \frac{\Delta_1^n(z_0,w_0)}{w_0} +\left( \frac{z_0}{w_0}+ \frac{\Delta_1^n(z_0,w_0)}{w_0}\right)\sum_{j=1}^\infty\left( \frac{-\Delta_2^n(z_0,w_0)}{w_0} \right)^{j} \text{ $\forall n\geq0$.}
\end{equation}
This expression makes sense for $|x|=\left|\frac{-\Delta_2^n(z_0,w_0)}{w_0}     \right|<1$, hence, in view of (\ref{eq: bounds on Delta in S}), for  $|w_0|>1$. Recall also that $|\sum_{j=1}^\infty x^{j}|=\frac{|x|}{1-x}\leq 2|x|$ if $|x|<\frac{1}{2}$.
Let $K\subset\hat\C$  be a compact set and suppose that $\frac{z_0}{w_0}$ takes values in $K$.  By  (\ref{eq: distance h0 h1}) and using (\ref{eq: bounds on Delta in S}),  for   any  $\epsilon>0$    there exists $M=M(K,\epsilon)$ such that 
\begin{equation}
\bigg|\frac{z_{2n}}{w_{2n}}-\frac{z_{0}}{w_{0}}\bigg|<\epsilon \text{ for $|w_0|>M$ and $\frac{z_0}{w_0}\in K$.}
\end{equation}
Consider the function  $G(z,w):=\frac{z}{w}$
Observe that
 $$
 G^{-1}(re^{i\theta})=\{(r_1e^{i\theta_1},r_2 e^{i\theta_2})\in\C^2: \frac{r_1}{r_2}=r, \theta=\theta_1-\theta_2 \}.
 $$ 
 Let $c\in \H$.
 By the shape of $W$ we have  that $G(W^{++})=\H$, that  $\epsilon:=\frac{1}{2}\dist_{\operatorname{spher}}(c, G(\partial W ))>0$, and that we can choose $Q=(z_0,w_0)\in W^{++} \in G^{-1}(c)$ such that $|w_0|$ is arbitrarily large.  
By taking a limit in $n$ in equation (\ref{eq: distance h0 h1}) and  on a sufficiently small polydisk centered at $Q$ we can ensure  that  $\dist_{\operatorname{spher}}(h_1,G)<\epsilon$, hence the claim follows by Rouch\'e's Theorem.
\end{proof}

The main Theorem is a direct consequence of Propositions~\ref{prop:existence of FC}, \ref{prop:conjugacy},     \ref{prop:geometry of Omega}, \ref{prop:limit set for Omega}.

\section{Appendix: Proof of Proposition \ref{prop: harmonic bounded outside analytic set}}
We split the proof of Proposition \ref{prop: harmonic bounded outside analytic set} over several lemmas.

\begin{defn}
 Let $E\subset\C^n$. A vector $v\in\C^n$ is called tangent to $E$ at a point $P\in\overline E$ if there exist a sequence of points $P_J \in E$ and real numbers $t_j > 0$ such that $P_i \to P$ and $t_j(P_j - P) \to v$ as $j\to\infty$.
The set of all such tangent vectors is the \emph{tangent
cone} to $E$ at $P$.
\end{defn}

The tangent cone is indeed a cone in $\C^n=T_P\C^n$. If the set $E$ is a $\mathcal{C}^1$-smooth manifold, the tangent cone coincides with the tangent space.

For complex analytic sets of dimension one, the following is a well known fact. For a proof, see \cite{Chirka1989}, Corollary on page 80.
\begin{lemma}\label{lem1}
    Let $H\subset\C^n$ be an analytic set of dimension one. For all $x\in H$ the tangent cone of $H$ at $x$ consists of a finite union of complex lines (whose number is not greater than the number of irreducible components of $H$ at $x$).
\end{lemma}

\begin{defn}
 Let $B\subset\C^n$ be a polidisc. The torus $\T$ with same center and same poliradius as $B$ is called its \emph{\v Silov boundary.} We will denote it by $\partial_{S}B$.
\end{defn}

The \v Silov boundary is a very general notion, for Banach algebras, but we will not need it here in all generality. For details, we refer to \cite{DSSST}, from page 325.

\begin{lemma}\label{lem2}
 Let $B$ be a polydisk, $\partial_{S}B$ be its \v Silov boundary, and  $u:U\to \R$ be a harmonic function defined on a neighbourhood $U$ of $\ov{B}$.
    Then
    $$\max_{\overline B} u\,=\,\max_{\partial_{S}B}u.$$
\end{lemma}
Recall that $\partial$ denotes the topological boundary.
\begin{proof}
    For every $P\in \partial B\setminus {\partial_{S}B}$ there is a horizontal Euclidean  disc $D$ through $P$ which is contained in $\partial B$ whose boundary is in ${\partial_{S}B}$. Being $u$ harmonic in a neighbourhood of $\overline B$, $u$ is harmonic on such a closed disc, hence its value at $P$ is less or equal to the maximum of $u$ at its boundary $\partial D\subset\partial_{S} B$. Hence
    $$\max_{\partial B} u\,=\,\max_{\partial_{S} B}u\,.$$
    If $P\in B$, we can find a disc through $P$ with boundary in $\partial B$ and repeat the argument, getting the conclusion
    $$\max_{\overline B} u\,=\,\max_{\partial B}u\,=\,\max_{\partial_{S}B}u\,.$$
\end{proof}

\begin{lemma}\label{lem:torofuoriH}
    Let $H\subset\C^2$ be an analytic set of dimension one.  Then  for every $ P\in H$ there exists an arbitrarily small torus  $\T_P$ centered in $P$ such that $\T_P\cap H=\emptyset$.
\end{lemma}
\begin{proof}
    Let $P\in H$. Consider the tangent cone $C_P$ of $H$ at $P$. By Lemma \ref{lem1}, $C_P$ is a finite set of directions $\alpha_1,\ldots,\alpha_k\in\hat\C$. Up to a rotation, we can suppose all directions to be in $\C$. Up to choosing $\eta>0$ small enough, we can ensure  that the polidisk  $B_\eta$  of poliradius $\eta$ centered in $P$ intersects only one connected component of  $H$. Moreover, by the definition of tangent cone, we can choose a small neighbourhood $K\subset\C$ of all $\alpha_j$ such that 
    $$
    H\cap U_\eta\ \subset\ \cup_{\alpha\in K}(P+(z,\alpha z))\,.
    $$
    We can suppose $K$ to be small enough that there is $0<b<1$ such that $K\cap\{\beta\in \C\ | \ |\beta|=b\}=\emptyset$.

    For any $0<\delta<\eta$, defining
    $$
    \T_P=\left\{|z-z_P|=\delta\right\}\times\left\{|w-w_P|=b\delta\right\}
    $$
    we have that $\T_P\subset U_\delta$ and if $(z,w)\in \T_P$, $\frac{w-w_P}{z-z_P}=\beta$ with $|\beta|=b$. So
    $$
    \T_P\cap H\ =\ T_P\cap U_\eta \cap H\ =\ \emptyset.
    $$
\end{proof}

\begin{proof}[Proof of Proposition \ref{prop: harmonic bounded outside analytic set}]
Let $K$ and $L$ be compacts sets as in the statement. For each $P\in H\cap K$, by Lemma \ref{lem:torofuoriH} there exists a torus $\T_P\subset L$ centered in $P$ such that $\T_P\cap H=\emptyset$. Each torus $\T_P$ is the \v Silov boundary of a polidisk $B_P$ centered in $P$. Since $\{B_P\}_{P\in H\cap K}$ is a covering of $H\cap K$, by compactness we can extract a finite covering $\{B_1,\ldots,B_k\}$.

There is a $\eta$-neighbourhood $U_\eta$ of $H$ such that $U_\eta\cap K\subset \cup B_j$. If the harmonic function $u$ satisfies $u\leq \alpha$ on $L\setminus U_\eta$ then it satisfies the same estimate on all tori $\T_j$, and by Lemma \ref{lem2} the same estimate holds on all $B_j$. Hence $u\leq \alpha$ on $K$.
\end{proof}

\bibliographystyle{amsalpha}
\bibliography{EscapingFC}
\end{document}